\documentclass[12pt]{article}
\usepackage{amssymb}
\usepackage{amsmath,amsthm}
\usepackage[latin1]{inputenc}
\usepackage[dvips]{graphicx}
\usepackage{hyperref}
\usepackage{enumerate}

\hypersetup{colorlinks=true}

\hypersetup{colorlinks=true, linkcolor=blue, citecolor=blue,urlcolor=blue}

 \newtheorem{remark}{Remark}

 \newtheorem{lemma}[remark]{Lemma}
 \newtheorem{theorem}[remark]{Theorem}
 
 \newtheorem{corollary}[remark]{Corollary}
  \newtheorem{claim}[remark]{Claim}
  \newtheorem{conjecture}[remark]{Conjecture}

\title{The metric dimension of strong product graphs}

\author{Juan A. Rodr\'{\i}guez-Vel\'{a}zquez$^{(1)}$, Dorota Kuziak$^{(1)}$,\\
Ismael G. Yero$^{(2)}$ and Jos\'e M. Sigarreta$^{(3)}$\\
$^{(1)}${\small Departament d'Enginyeria Inform\`atica i Matem\`atiques,}\\
{\small Universitat Rovira i Virgili,}  {\small Av. Pa\"{\i}sos
Catalans 26, 43007 Tarragona, Spain.} \\{\small
juanalberto.rodriguez\@@urv.cat, dorota.kuziak\@@urv.cat}
\\
$^{(2)}${\small Departamento de Matem\'aticas, Escuela Polit\'ecnica Superior de Algeciras}\\
{\small Universidad de C\'adiz,} {\small
Av. Ram\'on Puyol s/n, 11202 Algeciras, Spain.} \\ {\small
ismael.gonzalez\@@uca.es}\\
$^3${\small Facultad de  Matem\'{a}ticas,} {\small Universidad Aut\'onoma de Guerrero}
\\
{\small Carlos E. Adame 5, Col. La Garita, Acapulco, Guerrero,
M\'{e}xico }\\{\small josemariasigarretaalmira@hotmail.com}}

\begin{document}
\maketitle

\begin{abstract}

For an ordered subset $S = \{s_1, s_2,\dots s_k\}$ of vertices and a vertex $u$ in a connected graph $G$, the metric representation of $u$ with respect to $S$ is the ordered $k$-tuple $ r(u|S)=(d_G(v,s_1), d_G(v,s_2),\dots,$ $d_G(v,s_k))$, where $d_G(x,y)$ represents the distance between the vertices $x$ and $y$. The set $S$ is a metric generator for $G$ if every two different vertices of $G$ have distinct metric representations. A minimum metric generator is called a metric
basis for $G$ and its cardinality, $dim(G)$, the metric dimension of $G$. It is well known that the problem of finding the metric dimension of a graph is NP-Hard.  In this paper we obtain closed formulae and tight bounds for the metric dimension of strong product graphs.
\end{abstract}



\section{Introduction}

A {\em generator} of a metric space is a set $S$ of points in the space with the property that every point of the space is uniquely determined by its distances from the elements of $S$. Given a simple and connected graph $G=(V,E)$, we consider the metric $d_G:V\times V\rightarrow \mathbb{N}$, where $d_G(x,y)$ is the length of a shortest path between $x$ and $y$. $(V,d_G)$ is clearly a metric space. A vertex $v\in V$ is said to distinguish two vertices $x$ and $y$ if $d_G(v,x)\ne d_G(v,y)$. A set $S\subset V$ is said to be a \emph{metric generator} for $G$ if any pair of vertices of $G$ is distinguished by some element of $S$. A metric generator of minimum cardinality is called a \emph{metric basis}, and its cardinality the \emph{metric dimension} of $G$, denoted by $dim(G)$.

The concept of metric dimension was introduced by Slater in \cite{leaves-trees}, where the metric generators were called \emph{locating sets}, and studied independently by Harary and Melter \cite{harary}, where the metric generators were called \emph{resolving sets}. Applications of this invariant to the navigation of robots in networks are discussed in \cite{landmarks} and applications to chemistry in \cite{pharmacy1,pharmacy2}. This invariant was studied further in a number of other papers including for example \cite{pelayo1,chappell,chartrand,chartrand2,LocalMetric,yerocartpartres,CMWA}. Several variations of metric generators have been appearing in the literature, like those about resolving dominating sets \cite{brigham}, independent resolving sets \cite{chartrand1}, local metric sets \cite{LocalMetric}, resolving partitions \cite{chartrand2,yerocartpartres}, and strong metric generators \cite{strongDimensionCorona,seb}.

It was shown in \cite{Garey} that the problem of computing $dim(G)$ is NP-hard. This suggests
finding the metric dimension for special classes of graphs or obtaining good bounds on this
invariant. Metric basis have been studied, for instance, for digraphs \cite{Oellermann}, Cartesian product graphs \cite{pelayo1,yerocartpartres}, corona product graphs \cite{strongDimensionCorona,CMWA}, distance-hereditary graphs \cite{may-oellermann}, and Hamming graphs \cite{strong-hamming}.
In this paper we study the problem of finding exact values or sharp
bounds for the metric dimension of strong product graphs and express
these in terms of invariants of the factor graphs.

The strong product of two graphs $G=(V_1,E_1)$ and $H=(V_2,E_2)$ is the graph $G\boxtimes H=(V,E)$, such that $V=V_1\times V_2$ and two vertices $(a,b),(c,d)\in V$ are adjacent in $G\boxtimes H$ if and only if
\begin{itemize}
\item $a=c$ and $bd\in E_2$ or
\item $b=d$ and $ac\in E_1$ or
\item $ac\in E_1$ and $bd\in E_2$.
\end{itemize}
One of our tools will be a well-known result which states the relationship between the distance between two vertices in $G\boxtimes H$ and the  distances between the corresponding  vertices in the factor graphs.
\begin{remark}{\rm \cite{BookProduct}} Let $G$ and $H$ be two connected graphs. Then
$$d_{G\boxtimes H}((a,b),(c,d))=\max \{d_G(a,c), d_H(b,d)\}.$$
\end{remark}

\section{Results}

We begin with a general upper bound for the metric dimension of strong product graphs.

\begin{theorem}\label{generalUpperBound}
Let $G$ and $H$ be two connected graphs of order $n_1\ge 2$ and $n_2$, respectively. Then $$dim(G\boxtimes H)\le n_1\cdot dim(H) + n_2\cdot dim(G) - dim(G)\cdot dim(H).$$
\end{theorem}

\begin{proof}
Let $V_1=\{u_1,u_2,...,u_{n_1}\}$ and $V_2=\{v_1,v_2,...,v_{n_2}\}$ be the set of vertices of $G$ and $H$, respectively. Let $S=(V_1\times S_2)\cup (S_1\times V_2)$, where $S_1$ and $S_2$ are metric basis for $G$ and $H$, respectively. 
 Let $(u_i,v_j)$ and $(u_k,v_l)$ be two different vertices of $G\boxtimes H$. 
 Let $u_\alpha \in S_1$ such that $u_i,u_k$ are distinguished by $u_\alpha$ and let $v_\beta\in S_2$ such that $v_j,v_l$ are distinguished by $v_\beta$. If $i=k$, then $(u_i,v_j)$ and $(u_k,v_l)$ are distinguished by  $(u_i,v_{\beta}) \in (V_1\times S_2)\subset S$. Analogously, if $j=l$, then $(u_i,v_j)$ and $(u_k,v_l)$ are distinguished by  $(u_{\alpha},v_j)\in (S_1\times V_2)\subset S$. If $i\ne k$ and $j\ne l$, then we suppose that neither $(u_i,v_\beta)$ nor $(u_k,v_\beta)$ distinguishes the pair $(u_i,v_j), (u_k,v_l)$, \emph{i.e.,}
\begin{equation}
d_{G\boxtimes H}((u_i,v_j),(u_i,v_\beta)) = d_{G\boxtimes H}((u_k,v_l),(u_i,v_\beta))   \label{fori}
\end{equation}
and
\begin{equation}
d_{G\boxtimes H}((u_i,v_j),(u_k,v_\beta)) = d_{G\boxtimes H}((u_k,v_l),(u_k,v_\beta)).  \label{fork}
\end{equation}
By (\ref{fori}) we have $d_H(v_j,v_\beta) = \max \{d_G(u_k,u_i), d_H(v_l,v_\beta)\}$ and since $d_H(v_j,v_\beta)\ne d_H(v_l,v_\beta)$, we obtain that
\begin{equation}
d_H(v_j,v_\beta) = d_G(u_k,u_i).   \label{confori}
\end{equation}
Also, by (\ref{fork}) we have $d_H(v_l,v_\beta) = \max \{d_G(u_i,u_k), d_H(v_j,v_\beta)\}$ and since $d_H(v_j,v_\beta)\ne d_H(v_l,v_\beta)$, we obtain that
\begin{equation}
d_H(v_l,v_\beta) = d_G(u_i,u_k).  \label{confork}
\end{equation}

From (\ref{confori}) and (\ref{confork}) we have that $d_H(v_j,v_\beta) = d_H(v_l,v_\beta)$ which is a contradiction with the statement that $v_j,v_l$ are distinguished by $v_\beta$ in $H$.
\end{proof}

Since $K_{n_1}\boxtimes K_{n_2} \cong K_{n_1\cdot n_2}$ and for any complete graph $K_n$, $dim(K_n)=n-1$, we deduce
\begin{align*}
dim(K_{n_1}\boxtimes K_{n_2})  &= n_1\cdot n_2 - 1 \\
&=n_1(n_2-1)+n_2(n_1-1)-(n_1-1)(n_2-1)\\
& = n_1\cdot dim(K_{n_2}) + n_2\cdot dim(K_{n_1}) - dim(K_{n_1})\cdot dim(K_{n_2}).
\end{align*}
Therefore,  the above bound is tight. Examples of non-complete graphs where the above bound is attained can be derived from Theorem \ref{Generalizacompleto-por-G}.

Given  two vertices $x$ and $y$ in a connected graph $G=(V,E)$, the interval $I[x, y]$ between $x$ and $y$ is defined as the collection of all vertices which lie on some shortest $x-y$ path. Given a nonnegative integer $k$ we say that \emph{$G$ is self $k$-resolved} if for every two different vertices $x,y\in V$, there exists $w\in V$ such that
\begin{itemize}
\item $d_G(y,w)\ge k$ and $x \in  I[y,w]$ or
\item $d_G(x,w)\ge k$ and  $y \in  I[x,w]$.
\end{itemize}
For instance, the path graphs $P_n$ ($n\ge 2$) are self $\left\lceil \frac{n}{2}\right\rceil$-resolved, the two-dimensional grid graphs $P_n\square P_m$ are self $\left(\lceil\frac{n}{2}\rceil+\lceil\frac{m}{2}\rceil\right)$-resolved, the hypercube graphs $Q_k$ are self $k$-resolved and the pseudo sphere graphs $S_{k,r}$ ($k,r\ge 2$)  are self  $k$-resolved, where $S_{k,r}$ is a graph  defined as follows: we consider $r$ path graphs of order greater than or equal to $k+1$ and we identify one extreme of each one of the $r$ path graphs in one pole $a$ and all the other extreme vertices of the paths in a pole $b$. In particular, $S_{k,2}$ is a cycle graph.

\begin{theorem} \label{resolvedGraphs}
Let $H$ be a self $k$-resolved graph of order $n_2$ and let $G$ be a graph of diameter $D(G)<k$. Then
$$dim(G\boxtimes H)\le n_2\cdot dim(G).$$
\end{theorem}

\begin{proof}
Let $V_1=\{u_1,u_2,...,u_{n_1}\}$ and $V_2=\{v_1, v_2,...,v_{n_2}\}$ be the set of vertices of $G$ and $H$, respectively. Let $S_1$ be a metric generator for $G$. We will show that $S=S_1\times V_2$ is a metric generator for $G\boxtimes H$.  Let $(u_i,v_j),(u_r,v_l)$ be two different vertices of $G\boxtimes H$. We differentiate the following two cases.

{\bf Case 1.} $j=l$. Since  $i\ne r$ and  $S_1$ is a metric generator for $G$,  there exists $u\in S_1$ such that $d_G(u_i,u)\ne d_G(u_r,u)$. Hence,
$$d_{G\boxtimes H}((u_i,v_j),(u,v_j))=d_G(u_i,u)\ne d_G(u_r,u)=d_{G\boxtimes H}((u_r,v_j),(u,v_j)).$$

{\bf Case 2. } $j\ne l$. Since $H$ is self $k$-resolved, there exists $v\in V_2$ such that $d_H(v,v_l)\ge k$ and $v_j \in  I[v,v_l]$  or $d_H(v,v_j)\ge k$ and $v_l\in  I[v,v_j]$. Say $d_H(v,v_l)\ge k$ and $v_j \in  I[v,v_l]$.
In such a case, for every $u\in S$ we have,
\begin{align*}
d_{G\boxtimes H}((u_i,v_j),(u,v))&=\max\{d_G(u_i,u), d_H(v_j,v)\}\\
& < d_H(v,v_l) \\
&=\max\{d_G(u,u_r), d_H(v,v_l)\}\\
&=d_{G\boxtimes H}((u_r,v_l),(u,v)).
\end{align*}
Therefore, $S$ is a metric generator for $G\boxtimes H$.
\end{proof}

Now we derive some consequences of the above result.
\begin{corollary}\label{completo-camino-ciclo+otros}
Let $n_1\ge 2$ be an integer.
\begin{itemize}
\item  For any integer $n_2\ge 4$ such that $n_1-1 < \left\lfloor \frac{n_2}{2}\right\rfloor$, $$dim(P_{n_1}\boxtimes C_{n_2})\le n_2.$$

\item  Let $k\ge 2$ be an integer. For any  self $k$-resolved graph $H$ of order $n_2$, $$dim(K_{n_1}\boxtimes H)\le (n_1-1)n_2.$$
\end{itemize}
\end{corollary}

Given a vertex $v$ of a graph $G=(V,E)$, we denote by $N_G(v)$ the open neighborhood of $v$, {\it i.e.,} the set of neighbors of $v$, and by $N_G[v]$ the closed neighborhood of $v$, {\it i.e.,} $N_G[v]=N_G(v)\cup \{v\}$. Two  vertices $u$ and $v$ are false twins if $N_G(u)=N_G(v)$, while they are true twins if $N_G[u]=N_G[v]$.
Note that if two vertices $u$ and $v$ of a graph $G=(V,E)$ are (true or false) twins, then  $d_G(x,u)=d_G(x,v)$, for every $x\in V-\{u,v\}$.
We define the true twin  equivalence relation ${\cal R}$ on $V(G)$ as follows:
$$x {\cal R} y \leftrightarrow N_G[x]=N_G[y].$$
If the true twin equivalence classes are $U_1,U_2,...,U_t$, then every metric generator of $G$ must contain at least $|U_i|-1$ vertices from $U_i$ for each $i\in \{1,...,t\}$. Therefore, $$dim(G)\ge \sum_{i=1}^t(|U_i|-1)=n-t,$$ where $n$ is the order of $G$.


\begin{theorem}\label{Generalizacompleto-por-G}
Let $G$ and $H$ be two nontrivial connected graphs of order $n_1$ and $n_2$ having $t_1$ and $t_2$ true twin equivalent classes, respectively. Then $$dim(G\boxtimes H)\ge n_1n_2-t_1t_2.$$
Moreover, if $dim(G)=n_1-t_1$ and $dim(H)=n_2-t_2$, then  $$dim(G\boxtimes H)= n_1n_2-t_1t_2.$$
\end{theorem}

\begin{proof}
Let $U_1,U_2,...,U_{t_1}$ and $U'_1,U'_2,...,U'_{t_2}$ be the true twin equivalence classes of $G$ and $H$, respectively.
Since  each  $U_i$ (and $U'_j$) induces  a clique and its vertices have identical closed neighborhoods, $U_i\times U'_j$  induces  a clique in $G\boxtimes H$ and its vertices have identical closed neighborhoods, \emph{i.e.}, for every $a,c\in U_i$ and $b,d\in U'_j$,
\begin{align*}N_{G\boxtimes H}[(a,b)]&= \{(x,y): \; x\in N_G[a], y\in N_H[b]\}\\
&= \{(x,y): \; x\in N_G[c],y\in N_H[d]\}\\
&=N_{G\boxtimes H}[(c,d)].
\end{align*}
Hence, $V(G)\times V(H)$ is partitioned as $$V(G)\times V(H)=\bigcup_{j=1}^{t_2}\left(\bigcup_{i=1}^{t_1} U_i\times U'_j \right),$$
where   $U_i\times U'_j$  induces a clique in $G\boxtimes H$ and its vertices have identical closed neighborhoods. Therefore,  the metric
dimension of $G\boxtimes H$  is at least  $\sum_{j=1}^{t_2}\left(\sum_{i=1}^{t_1}(|U_i||U_j| -1)\right)=n_1n_2-t_1t_2.$

Finally, if $dim(G)=n_1-t_1$ and $dim(H)=n_2-t_2$, then the above bound and Theorem \ref{generalUpperBound} lead to $dim(G\boxtimes H)= n_1n_2-t_1t_2.$
\end{proof}

As an example of non-complete graph $G$ of order $n$ having $t$ true twin equivalent classes, where $dim(G)=n-t$, we take $G=K_1+(\displaystyle\bigcup_{i=1}^lK_{r_i})$, $r_i\ge 2$, $l\ge 2$. In this case $G$ has $t=l+1$ true twin equivalent classes, $n=1+\sum_{i=1}^lr_i$ and $dim(G)=\sum_{i=1}^l(r_i-1)=n-t$.



\begin{corollary}\label{GeneralizCoroInferior}
Let $H$ be a graph of order $n_2$. Let $G$  be a nontrivial connected graph  of order $n_1$   having $t_1$ true twin equivalent classes. Then
$$dim(G\boxtimes H)\ge  n_2(n_1-t_1).$$
\end{corollary}

Theorem \ref{resolvedGraphs} and Corollary \ref{GeneralizCoroInferior} lead to the following result.
\begin{theorem}\label{Th-dim=n(n1-1)}
Let $H$ be a self $k$-resolved graph of order $n_2$ and let  $G$  be a nontrivial connected graph  of order $n_1$   having $t_1$ true twin equivalent classes and  diameter $D(G)<k$. If $dim(G)=n_1-t_1$, then
$$dim(G\boxtimes H)= n_2(n_1-t_1).$$
\end{theorem}

\begin{lemma}\label{Lemma2resolved}
A nontrivial connected graph is self $2$-resolved if and only if it does not have true twin vertices.
\end{lemma}

\begin{proof} Necessity. Let $G$ be a  $2$-resolved graph. Let $x$ and $y$ be two adjacent vertices in $G$. Without loss of generality, we take $w\in V(G)$ such that $2\le k=d_G(x,w)$ and $y\in I[x,w]$. So, there exists a shortest path $x,y,u_2,...,u_{k-1},w$ from $x$ to $w$ and, as a consequence, $u_2\in N_G[y]$ and $u_2\not\in N_G[x]$. Therefore, $G$ does not have true twin vertices.

Sufficiency. If for every  $u,v\in V(G)$, $N_G[u]\ne N_G[v]$, then for each pair of adjacent vertices $x$ and $y$, there exists $w\in V(G)-\{x,y\}$ such that $d_G(x,w)=2$ and $y\in I[x,w]$ or $d_G(y,w)=2$ and $x\in I[y,w]$. On the other hand, if $d_G(u,v)\ge 2$, then for $w=u$ we have $d_G(v,w)\ge 2$ and $u\in I[v,w]$. Therefore, $G$ is self $2$-resolved.
\end{proof}

By Lemma \ref{Lemma2resolved} we deduce the following consequence of Theorem \ref{Th-dim=n(n1-1)}.
\begin{corollary} Let $H$ be a connected graph  of order $n_2\ge 3$.
 If $H$  does not have  true twin vertices and  $n_1\ge 2$, then
$dim(K_{n_1}\boxtimes H)= n_2(n_1-1).$
\end{corollary}

The following   remark emphasizes some particular cases of the above result.
\begin{remark}\label{completo-camino-ciclo} Let $n_1\ge 2$ be an integer.
\begin{itemize}
\item  For any tree $T$ of order $n_2\ge 3$,
$dim(K_{n_1} \boxtimes T)= n_2(n_1-1).$

\item  For any $n_2\ge 4$,
$dim(K_{n_1} \boxtimes C_{n_2})= n_2(n_1-1).$

\item  For any hypercube $Q_{r}=\underbrace{K_2\square \cdots \square K_2}_{r}$, $r\ge 2$,
$dim(K_{n_1} \boxtimes Q_{r})= 2^r(n_1-1).$

\item For any integers $m,n\ge 2$,
$dim(K_{n_1} \boxtimes (P_{n}\Box P_{m}))= n\cdot m\cdot (n_1-1).$
\end{itemize}
\end{remark}

Now we proceed to study  the strong product of path graphs.

\begin{theorem}
For any integers $n_1$ and $n_2$ such that $2\le n_1< n_2$,
$$\left\lfloor\frac{n_1+n_2-2}{n_1-1}\right\rfloor\le dim(P_{n_1} \boxtimes P_{n_2})\le  \left\lceil\frac{n_1+n_2-2}{n_1-1}\right\rceil.$$
\end{theorem}

\begin{proof}
Let $V_1=\{u_1, u_2,...,u_{n_1}\}$ and $V_2=\{v_1, v_2,...,v_{n_2}\}$ be the set of vertices of $P_{n_1} $ and $ P_{n_2}$, respectively. With the above notation we suppose that two consecutive vertices of $V_i$ are adjacent, $i\in \{1,2\}$.

Let $\alpha=\left\lceil \frac{n_2-1}{n_1-1} \right\rceil -1$. We define the set $S$ of cardinality $\left\lceil\frac{n_1+n_2-2}{n_1-1}\right\rceil$ as follows:
$$S=\{(u_1,v_1), (u_{n_1},v_{n_1}),(u_1,v_{2(n_1-1)+1}),(u_{n_1},v_{3(n_1-1)+1}),...,(u_{1},v_{\alpha(n_1-1)+1}),(u_{n_1},v_{n_2})\}$$ if  $\left\lceil \frac{n_2-1}{n_1-1} \right\rceil$ is odd, and
$$S=\{(u_1,v_1), (u_{n_1},v_{n_1}),(u_1,v_{2(n_1-1)+1}),(u_{n_1},v_{3(n_1-1)+1}),...,(u_{n_1},v_{\alpha(n_1-1)+1}),(u_{_1},v_{n_2})\}$$ if $\left\lceil \frac{n_2-1}{n_1-1} \right\rceil$ is even. We will show that $S$ is a metric generator for $P_{n_1} \boxtimes P_{n_2}$. Let $(u_i,v_j),(u_k,v_l)$ be two different vertices of $P_{n_1} \boxtimes P_{n_2}$. We differentiate two cases.

Case 1. $j=l$. We suppose, without loss of generality, that $i<k$.  If $j\in \{1,...,n_1\}$ and $d_{P_{n_1} \boxtimes P_{n_2}}((u_i,v_j),(u_{n_1},v_{n_1}))=d_{P_{n_1} \boxtimes P_{n_2}}((u_k,v_j),(u_{n_1},v_{n_1}))$, then from $\max\{n_1-i,n_1-j\}=\max\{n_1-k,n_1-j\}$ we have $n_1-j\ge n_1-i>n_1-k$. Hence, $j<k$ and, as a consequence,
\begin{align*}
d_{P_{n_1} \boxtimes P_{n_2}}((u_i,v_j),(u_{1},v_{1}))&=\max\{i-1,j-1\}\\
&<k-1\\
&=\max\{k-1,j-1\}\\
&=d_{P_{n_1} \boxtimes P_{n_2}}((u_k,v_j),(u_{1},v_{1})).
\end{align*}
Thus, if $j\in \{1,...,n_1\}$, then we deduce $r((u_i,u_j)|S)\ne r((u_k,u_j)|S)$.

An analogous procedure can be used to show that for $$j\in \{t(n_1-1)+1,...,(t+1)(n_1-1)+1\},$$ where $t\in \{1,..,\alpha-1\}$, and for $j\in \{\alpha(n_1-1)+1,...,n_2\}$, it follows $r((u_i,u_j)|S)\ne r((u_k,u_j)|S)$.

Case 2. $j\ne l$. We suppose, without loss of generality, that $j<l$ and we differentiate two subcases.

Subcase 2.1. $l<n_1$. Since $(u_1,v_1),(u_{n_1},v_{n_1})\in S$, we only must consider the case when $$d_{P_{n_1} \boxtimes P_{n_2}}((u_i,v_j),(u_1,v_1))=d_{P_{n_1} \boxtimes P_{n_2}}((u_k,v_l),(u_1,v_1))$$ and $$d_{P_{n_1} \boxtimes P_{n_2}}((u_i,v_j),(u_{n_1},v_{n_1}))=d_{P_{n_1} \boxtimes P_{n_2}}((u_k,v_l),(u_{n_1},v_{n_1})).$$ In such a situation, since $j<l$, we have $k<i$. Hence,
\begin{align*}
d_{P_{n_1} \boxtimes P_{n_2}}((u_i,v_j),(u_{1},v_{n_2})&=\max \{i-1,n_2-j\}\\
&>\max\{k-1,n_2-l\}\\
&=d_{P_{n_1} \boxtimes P_{n_2}}((u_k,v_l),(u_{1},v_{n_2}).
\end{align*}
So, if $(u_1,v_{n_2})\in S$, then $r((u_i,v_j)|S)\ne r((u_k,v_l)|S)$. Moreover, if $(u_1,v_{n_2})\not\in S$, then $(u_1,v_{2n_1-1})\in S$. Hence, from
\begin{align*}d_{P_{n_1} \boxtimes P_{n_2}}((u_i,v_j),(u_{1},v_{2n_1-1})&=\max \{i-1,2n_1-1-j\}\\
&=2n_1-1-j\\
&>2n_1-1-l\\
&=\max\{k-1,2n_1-1-l\}\\
&=d_{P_{n_1} \boxtimes P_{n_2}}((u_k,v_l),(u_{1},v_{2n_1-1}),
\end{align*}
we have $r((u_i,v_j)|S)\ne r((u_k,v_l)|S)$.

Subcase 2.2. $l\ge n_1$. We have,
\begin{align*}d_{P_{n_1} \boxtimes P_{n_2}}((u_k,v_l),(u_1,v_1))&=\max \{k-1,l-1\}\\
&=l-1\\
&>\max \{i-1,j-1\}\\
&=d_{P_{n_1} \boxtimes P_{n_2}}((u_i,v_j),(u_1,v_1)).
\end{align*}
Thus, in this case  $r((u_i,v_j)|S)\ne r((u_k,v_l)|S)$ as well.

We conclude that $S$ is a metric generator for $P_{n_1} \boxtimes P_{n_2}$ and, as a consequence, the upper bound follows.

We will show that $dim(P_{n_1} \boxtimes P_{n_2})\ge \left\lfloor\frac{n_1+n_2-2}{n_1-1}\right\rfloor$ by contradiction. Let $n_2-1=x(n_1-1)+y$, where $n_1-1> y\ge 0 $. Now we suppose that there exists a metric generator for $P_{n_1} \boxtimes P_{n_2}$, say $S'$, of cardinality $x$. Note that a vertex $(u_r,v_t)\in S'$ distinguishes two vertices $(u_1,v_j), (u_2,v_j)$ if and only if  $|t-j|<r-1$. Analogously, a vertex $(u_r,v_t)\in S'$ distinguishes two vertices $(u_{n_1-1},v_{j'}), (u_{n_1},v_{j'})$ if and only if  $|t-j'|<n_1-r$. Hence, a vertex of $(u_r,v_t)\in S'$ distinguishes, at most, $2n_1-3$ pairs of vertices of the form $(u_1,v_j), (u_2,v_j)$ or
$(u_{n_1-1},v_{j'}), (u_{n_1},v_{j'})$. Thus, if $S'$ is a metric generator, then $2n_2-x\le (2n_1-3)x$  and, as a consequence, $n_2-1\le x(n_1-1)-1$, a contradiction.
\end{proof}


\begin{conjecture}
For any integers $n_1$ and $n_2$ such that $2\le n_1< n_2$,
 $$ dim(P_{n_1} \boxtimes P_{n_2}) =  \left\lceil\frac{n_1+n_2-2}{n_1-1}\right\rceil.$$
\end{conjecture}

\begin{theorem}
For any  integer $n\ge 2$, $dim(P_{n} \boxtimes P_{n})=3$.
\end{theorem}

\begin{proof}
Let $V=\{v_1,v_2,...,v_n\}$ be the set of vertices of $P_n$. Now, with the above notation,  we suppose that two consecutive vertices of $V$ are adjacent. We will  show that
$$S'=\{(u_1,v_1),(u_n,v_1),(u_n,v_n)\}$$
is a metric generator for $P_{n} \boxtimes P_{n}$. Let $(u_i,v_j),(u_k,v_l)$ be two different vertices of $P_{n} \boxtimes P_{n}$. We only must consider the case when $d_{P_{n} \boxtimes P_{n}}((u_i,v_j),(u_1,v_1))=d_{P_{n} \boxtimes P_{n}}((u_k,v_l),(u_1,v_1))$ and $d_{P_{n} \boxtimes P_{n}}((u_i,v_j),(u_{n},v_{n}))=d_{P_{n} \boxtimes P_{n}}((u_k,v_l),(u_{n},v_{n}))$. In such a case, if $j<l$, then  $k<i$ and, as a consequence,
\begin{align*}
d_{P_{n} \boxtimes P_{n}}((u_i,v_j),(u_{n},v_{1})&=\max \{n-i,j-1\}\\
&<\max\{n-k,l-1\}\\
&=d_{P_{n} \boxtimes P_{n}}((u_k,v_l),(u_{n},v_{1})).
\end{align*}
Analogously, if $j>l$, then we have $d_{P_{n} \boxtimes P_{n}}((u_i,v_j),(u_{n},v_{1})>d_{P_{n} \boxtimes P_{n}}((u_k,v_l),(u_{n},v_{1})).$
We conclude that $S'$ is a metric generator for $P_{n} \boxtimes P_{n}$ and, as a consequence, $dim(P_{n} \boxtimes P_{n})\le 3$. In order to show that $dim(P_{n} \boxtimes P_{n})\ge 3$, we suppose that there exists a metric generator for $P_{n} \boxtimes P_{n}$ of cardinality two. Since $(0,0)$ is not a possible distance vector, and the diameter of $P_{n} \boxtimes P_{n}$ is $n-1$, there are  $n^2-1$ possible distance vectors, but the order of  $P_{n} \boxtimes P_{n}$ is $n^2$, a contradiction. So, $dim(P_{n} \boxtimes P_{n})\ge 3$ and, as a consequence, $dim(P_{n} \boxtimes P_{n})= 3$.
\end{proof}

The following claim will be useful in the proof of Theorem \ref{caminociclon_2}.
\begin{claim} \label{distancescycles}
Let $C$ be a cycle graph. If $x,y,u$ and $v$ are vertices of $C$ such that $x$ and $y$ are adjacent and $d_C(u,x)=d_C(v,x)$, then $d_C(u,y)\neq d_C(v,y)$.
\end{claim}

\begin{theorem}\label{caminociclon_2}
For any integers $n_1$ and $n_2$ such that $\displaystyle\frac{n_1-1}{2}\ge \left\lfloor \frac{n_2}{2} \right\rfloor \ge 2$, $$dim(P_{n_1}\boxtimes C_{n_2})\le n_1.$$
\end{theorem}

\begin{proof}
Let $V_1=\{u_0,u_1,...,u_{n_1-1}\}$ and $V_2=\{v_0,v_1,...,v_{n_2-1}\}$ be the set of vertices of $P_{n_1}$ and $C_{n_2}$, respectively. Here we suppose that  $v_0$ and $v_{n_2-1}$ are adjacent vertices in $C_{n_2}$ and, with the above notation, two consecutive vertices of $V_i$ are adjacent, $i\in \{1,2\}$. Let $S$ be the set of vertices of $P_{n_1}\boxtimes C_{n_2}$ of the form $(u_i,v_i)$, where the subscript of the second component is taken modulo $n_2$. We will show that $S$ is a metric generator for $P_{n_1}\boxtimes C_{n_2}$. To begin with, we consider two different vertices $(u_i,v_j)$ and $(u_k,v_l)$  of $P_{n_1}\boxtimes C_{n_2}$. First we consider the case $i=k$ and we suppose, without loss of generality, that $j<l$. Now, if
$d_{P_{n_1}\boxtimes C_{n_2}}((u_i,v_j),(u_i,v_i))=d_{P_{n_1}\boxtimes C_{n_2}}((u_i,v_l),(u_i,v_i))$, then $d_{C_{n_2}}(v_j,v_i)=d_{C_{n_2}}(v_l,v_i)$. So, for $i=0$, Claim \ref{distancescycles} leads to
\begin{align*}d_{P_{n_1}\boxtimes C_{n_2}}((u_0,v_j),(u_1,v_1))&=\max\{1,d_{C_{n_2}}(v_j,v_1)\}\\
&\ne \max\{1,d_{C_{n_2}}(v_l,v_1)\}\\
&=d_{P_{n_1}\boxtimes C_{n_2}}((u_0,v_l),(u_1,v_1)).
\end{align*}
Analogously, for $i\ne 0$, Claim \ref{distancescycles} leads to
\begin{align*}d_{P_{n_1}\boxtimes C_{n_2}}((u_i,v_j),(u_{i-1},v_{i-1}))&=\max\{1,d_{C_{n_2}}(v_j,v_{i-1})\}\\
&\ne \max\{1,d_{C_{n_2}}(v_l,v_{i-1})\}\\
&=d_{P_{n_1}\boxtimes C_{n_2}}((u_i,v_l),(u_{i-1},v_{i-1})).
\end{align*}
Hence, $r((u_i,v_j)|S)\ne r((u_i,v_l)|S)$. Now we consider the case $i\ne k$. We suppose, without loss of generality, that $i<k$. If $k\le \left\lfloor \frac{n_2}{2} \right\rfloor$, then $n_1-1-i>\left\lfloor \frac{n_2}{2} \right\rfloor=D(C_{n_2})$. Thus,
\begin{align*}d_{P_{n_1}\boxtimes C_{n_2}}((u_i,v_j),(u_{n_1-1},v_{n_1-1}))&=\max\{d_{P_{n_1}}(u_i,u_{n_1-1}),d_{C_{n_2}}(v_j,v_{n_1-1})\}\\
&=\max \{n_1-1-i, d_{C_{n_2}}(v_j,v_{n_1-1}\}\\
&>  \max\{n_1-1-k ,d_{C_{n_2}}(v_l,v_{n_1-1})\}\\
&=d_{P_{n_1}\boxtimes C_{n_2}}((u_k,v_l),(u_{n_1-1},v_{n_1-1})).
\end{align*}
Moreover, if $k> \left\lfloor \frac{n_2}{2} \right\rfloor$, then
\begin{align*}d_{P_{n_1}\boxtimes C_{n_2}}((u_i,v_j),(u_{0},v_{0}))&=\max\{d_{P_{n_1}}(u_i,u_{0}),d_{C_{n_2}}(v_j,v_{0})\}\\
&=\max \{i, d_{C_{n_2}}(v_j,v_{0})\}\\
&<  \max\{k ,d_{C_{n_2}}(v_l,v_{0})\}\\
&=d_{P_{n_1}\boxtimes C_{n_2}}((u_k,v_l),(u_{0},v_{0})).
\end{align*}
Hence, $r((u_i,v_j)|S)\ne r((u_k,v_l)|S)$. Therefore, the set  $S$ of cardinality $n_1$ is a metric generator for $P_{n_1}\boxtimes C_{n_2}$.
\end{proof}


\end{document}